\documentclass[11pt]{amsart}

\usepackage{amsfonts,amsmath,amssymb}
\usepackage{enumerate}
\usepackage[pdftex,colorlinks]{hyperref}
\usepackage{tikz-cd}

\numberwithin{equation}{section}

\newtheorem{theorem}{Theorem}[section]
\newtheorem{lemma}[theorem]{Lemma}
\newtheorem{proposition}[theorem]{Proposition}

\newtheorem{definition}{Definition}[section]
\newtheorem{corollary}[theorem]{Corollary}
\newtheorem{remark}[theorem]{Remark}

\newcommand{\cl}[1]{\mathcal{#1}} 
\newcommand{\bb}[1]{\mathbb{#1}}
\newcommand{\sca}[1]{\left(#1\right)} 
\newcommand{\dua}[1]{\langle#1\rangle}
\newcommand{\nor}[1]{\left\Vert #1\right\Vert}

\def\eps{\varepsilon}
\def\gl{\lambda}
\def\gg{\gamma}
\def\go{\omega}
\def\ga{\alpha}

\begin{document}

\title[Homomorphisms of $L^1$ algebras and Fourier algebras]{Homomorphisms of $L^1$ algebras and Fourier algebras}

\author{M. Anoussis,  G. K. Eleftherakis, A. Katavolos}

\address{M. Anoussis\\Department of Mathematics\\ University of the Aegean\\
832 00 Karlovassi, Greece }\email{mano@aegean.gr}

\address{G. K. Eleftherakis\\Department of Mathematics\\ Faculty of Sciences\\
University of Patras\\265 00 Patra, Greece }
\email{gelefth@math.upatras.gr}

\address{A. Katavolos\\Department of Mathematics\\
National and Kapodistrian University of Athens\\
1578 84 Athens, Greece }\email{akatavol@math.uoa.gr}


\begin{abstract}   
We investigate conditions for the extendibility of continuous algebra homomorphisms 
$\phi$ from the Fourier algebra $A(F)$ of a locally compact group $F$ to  the Fourier-Stieltjes  algebra $B(G)$ 
of a locally compact group $G$ to maps between the corresponding $L^\infty$ algebras 
which are weak* continuous. 
When $\phi$ is completely bounded and $F$ is amenable, it is induced by a piecewise affine
map $\alpha: Y\to F$ where $Y\subseteq G$. 
We show that extendibility of $\phi$ is equivalent to $\alpha$ being an {\em open} map. 
\\
We also study the dual problem for contractive homomorphisms \\ 
$\phi: L^1(F)\to M(G)$.   
We show that $\phi$ induces a w* continuous homomorphism between the von Neumann 
algebras of the groups if and only if the naturally associated map $\theta$ 
(Greenleaf [1965], Stokke [2011]) is a {\em proper} map. 
\end{abstract}

\maketitle

\section{Introduction}\label{10000}

In this paper we investigate conditions under which homomorphisms between algebras associated to 
locally compact groups extend to von Neumann  algebras containing them.    

If $G$ is a locally compact group, the group algebra $L^1(G)$ as well as the measure algebra $M(G)$ act on $L^2(G)$
by convolution; their weak* closure is the von Neumann algebra of the group, $VN(G)$. 
It is thus a natural problem to determine conditions under which an algebra homomorphism 
from $L^1(F)$ to $M(G)$ induces a w*-continuous map between the corresponding von Neumann algebras.

The Fourier and Fourier-Stieltjes algebras of $G$ also act on $L^2(G)$ by (pointwise) multiplication, and their 
w* closure is the multiplication algebra of $L^\infty(G)$.   
Hence the corresponding problem about algebra homomorphisms from $A(F)$ to $B(G)$ is also natural. 

The Fourier and Fourier-Stieltjes algebras of a locally compact group $G$ were introduced
by Eymard in \cite{eym} and have  since been a main object of study in the area of 
Noncommutative Harmonic Analysis.

Recall that the Fourier-Stieltjes algebra $B(G)$ 
is the set of all coefficient functions 
$s\to\sca{\pi(s)\xi,\eta}, \ (\xi,\eta\in H_\pi)$ defined by unitary representations 
$(\pi, H_\pi)$ of $G$, while the Fourier algebra $A(G)$ of $G$ consists of the coefficients 
of the left regular representation $s\to\gl_s$ on $L^2(G)$, given by $\gl_s\xi(t):=\xi(s^{-1}t)$.

When $G$ is abelian, the Fourier algebra is isometrically isomorphic  to the group algebra  $L^1(\hat G)$
of the dual group  $\hat G$  of $G$ and the Fourier-Stieltjes algebra  is isometrically isomorphic to the
measure algebra $M(\hat G)$. These isomorphisms are realised through the Fourier transform.

An important problem in Harmonic Analysis is the determination of all
homomorphisms $\phi: A(F)\rightarrow B(G)$, for locally compact groups $F$ and $G$.

A bounded homomorphism  $\phi: A(F)\rightarrow B(G)$  always takes the form
$ \phi(u) =\phi_\ga(u):=\chi_Y(u\!\circ\!\ga)$ for an open subset $Y\subseteq G$ and a continuous map $\ga:Y\to F$.

Cohen \cite{co} completely characterized homomorphisms $L^1(F) \rightarrow M(G)$ 
for locally compact {\em abelian}
groups in terms of piecewise affine maps between the dual groups $\hat G$  and $\hat F$. 
In view of the above remark,
it follows that the result of Cohen provides a complete
characterisation of  homomorphisms $A(F) \rightarrow B(G)$ for abelian groups in terms of piecewise affine maps 
$\alpha: G \rightarrow F$.
To obtain his result he used a characterization of the idempotents in $B(G)$ which he proved in \cite{c}. 
This result revealed the fundamental role of the open coset ring of the group $G$.
Host \cite{h} generalized the characterization of  idempotents of $B(G)$ to general locally compact groups.

Ilie and Spronk used the natural operator space structure of the Fourier algebra, as the predual of the von 
Neumann algebra of the group,  to prove that  if $F$ is amenable,  a completely bounded homomorphism
$\phi:A(F) \rightarrow B(G)$ is induced by a continuous piecewise affine map $\alpha$ as above
\cite{IS, daws}. This generalises the result of Cohen.

Pham  in  \cite{ph}, characterized the contractive homomorphisms $\phi: A(F)\rightarrow B(G)$.
 However, the general case remains open.

In the first part of this paper we  prove that if $F$ and $G$ are locally compact second countable groups, 
and the homomorphism $\phi_\ga$  is induced by  a mixed continuous piecewise affine map $\alpha$ 
(see Definition \ref{1111}), then $\phi_\ga$ extends   to a w*-w* continuous map from $L^{\infty}(F)$ to 
$L^{\infty}(G)$ if and only if $\alpha$  is an open map.  For  amenable $F$, the maps $\phi_\ga$ that we 
consider strictly include those covered by  the results of Ilie and Spronk \cite{IS}.  

Note, for comparison, that  Ilie and Stokke characterised in \cite{IST} the  maps from $A(F)$ to $B(G)$ 
induced by piecewise affine maps that extend to maps from $B(F)$ to $B(G)$ which are continuous 
for the respective w* topologies.

\medskip

In the second part of the paper we   study the dual problem: we 
provide necessary and sufficient   
conditions under which contractive homomorphisms from $L^1(F)$ to the measure algebra 
$M(G)$ induce   homomorphisms
between the corresponding von Neumann algebras  which are weak* continuous.

We note that for $F$ and $G$ abelian, the homomorphisms $\phi: L^1(F)\to M(G)$ are characterized 
by the result of Cohen mentioned above. 

To motivate our approach to the dual problem, consider  
 a continuous homomorphism  $\theta: F\rightarrow G$ between abelian groups $F$ and $G$. 
Then the map $\phi:  L^1 (F ) \to M(G)$  defined by  $\phi(f ) =\theta_* f$ (considering $f$ as a measure) is a
continuous homomorphism. Taking Fourier transforms we obtain a homomorphism $A(\hat{F})\to B(\hat{G})$, 
of the form $u \mapsto u\circ \hat{\theta}$. Then it follows from the results of the first part that this map 
extends to a  w*-w* continuous map from $L^{\infty}(\hat F)$ to $L^{\infty}(\hat G)$ if and only if
$\hat \theta$  is an open map. But it is known \cite[Theorem 8]{dr} that the map 
$\hat \theta $ is open if and only if $\theta$ 
is  proper. Hence,  taking  Fourier transforms again, 
we conclude that the map $\phi$ induces a   w*-continuous *-homomorphism
$VN(F) \to VN(G)$ if and only if $\theta$  is a proper map. See Section \ref{3000} for details.

This result suggests that a similar condition should hold 
for general locally compact  groups.

The contractive homomorphisms $\phi: L^1(F)\to M(G)$ for locally compact groups $F$ and $G$  
were characterized by Greenleaf \cite {green}.
A more detailed description is given by Stokke  \cite{sto2011}.

The structure of such  homomorphisms  is more
complicated than  in the above special case. 
Nevertheless, we prove in Section \ref{4} that there still is a properness condition
which is necessary and sufficient for the existence  of a w*-continuous 
induced map.

As a consequence, we give a new proof of a result due to Stokke \cite{sto2012}, that the properness 
condition mentioned above is equivalent to the w*-continuity of the extended map $\varphi: M(F) \to M(G)$.

Our approach relies on operator algebra methods and exploits the duality between
the Fourier algebra and the von Neumann algebra of a group. 

\medskip
\noindent {\bf Notation } We use $\sca{\xi,\eta}$ for the scalar product in a Hilbert space (linear in $\xi$) and 
$\dua{x^*,x}$ for the bilinear pairing between a vector $x$ in a Banach space $X$ and a linear functional 
$x^*$ in its dual. The symbol $\mu\ll\nu$ denotes absolute continuity for measures $\mu$ and $\nu$ (i.e. every 
$\nu$-null set is $\mu$-null). For a locally compact group $G$, we denote by $g\mapsto \gl_g$ or $\gl_g^G$ 
its left regular representation and by $\mu\mapsto \lambda^G(\mu), \ \mu\in M(G)$ 
its `integrated form' (see section \ref{4}).


\section{Homomorphisms between Fourier Algebras which extend to the corresponding masas.}
\label{20000}

Let $G$ and $F$  be  locally compact  second countable groups with Haar measures $\mu$ and $\nu$ respectively. 
We are concerned with conditions on a homomorphism 
$\phi: A(F)\to B(G)$ which allow extension to a mapping from $L^\infty(F)$ to $L^\infty(G)$.
Consider a bounded homomorphism  $\phi: A(F)\rightarrow B(G)$  and let 
$$Y=\{t\in G: \exists u\in A(F), \;\;\phi(u)(t)\neq 0\}.$$ 
Then $Y$ is an open set and there exists a continuous map 
$\alpha: Y\to F$, such that for all $u\in A(F),$ we have
 $$
\phi (u)(t)=\phi_\alpha(u)(t):=\left\{\begin{array}{ll} u(\ga (t)), & t\in Y\\  
0, & t\in G\setminus Y \end{array}\right.
$$ 
see \cite[p. 101]{kanlau}.

\begin{proposition} \label{pr21}
If $\phi$ extends to a w*-continuous map from $L^\infty(F)$ to $L^\infty(G)$, 
then the extension is necessarily  a $*$-homomorphism, and in fact is given 
on bounded measurable functions by the same formula as above.
It follows that the measure $\alpha_*\mu$ defined on Borel subsets $E \subseteq F$ 
by $\alpha _*\mu(E)= \mu(\alpha^{-1}(E ))$  satisfies $\ga_*\mu\ll\nu$. 

Conversely if $\ga_*\mu\ll\nu$ then the above formula defines a necessarily w*-continuous homomorphism 
$\tilde\phi: L^\infty(F)\to L^\infty(G)$ which extends $\phi$.
\end{proposition}

\proof Denote the extension by $\tilde\phi$. Note first that the formula 
$$
\tilde\phi (u)(t)=\left\{\begin{array}{ll} u(\ga (t)), & t\in Y\\  
0, & t\in G\setminus Y\end{array}\right.
$$
is valid for all $u\in C_0(F)$. Indeed by the density of $A(F)$ in $C_0(F)$ there is a 
sequence $(u_n)$ in $A(F)$ converging uniformly  to $u$.
 
Now let $V\subseteq F$ be an open set. 
There exists a bounded sequence $(u_n)\subseteq C_0(F)$ 
of functions supported in $V$ such that 
$u_n(s)\rightarrow \chi_V(s)$ for all $s\in F$, {where  $\chi_V$ denotes the characteristic function
of $V$}.   
By dominated convergence this sequence converges to 
$\chi_V$ in the weak-* topology. Since $\tilde\phi$ is w*-continuous, we have that 
\begin{equation}\label{typ}
\tilde\phi (\chi_V)(t)=\left\{\begin{array}{ll} \chi_V(\ga (t)), & \text{for a.a. } t\in Y\\  
0, &  \text{for a.a. } t\in G\setminus Y\end{array}\right.
\end{equation} 
The collection of all Borel sets $V\subseteq F$ for which $\tilde\phi (\chi_V)$
satisfies (\ref{typ}) is easily seen to be  a $\sigma$-algebra. 
Since it contains the open sets, it  is equal to the Borel $\sigma$-algebra.

Therefore 
\begin{equation}\label{typ2}
\tilde\phi (u)(t)=\left\{\begin{array}{ll} u(\ga (t)), &  \text{for a.a. } t\in Y\\  
0, &  \text{for a.a. } t\in G\setminus Y\end{array}\right.
\end{equation}  
holds whenever $u$ is the characteristic function of a Borel set, and hence 
(by linearity and $\nor{\cdot}_\infty$-continuity)  it holds for all functions in $L^\infty(F)$.

It is now obvious that $\tilde\phi$ is a *-homomorphism.  

Now if $E\subseteq F$ is any Borel set, $\tilde\phi(\chi_E)=\chi_{\ga^{-1}(E)}$ must be in $L^\infty(G)$. 
It follows that if $\nu(E)=0$ then $\mu(\ga^{-1}(E))$ must vanish. This shows that  $\ga_*\mu\ll\nu$. 
\smallskip

Conversely if $\ga_*\mu\ll\nu$ then (\ref{typ2}) clearly defines a *-homomorphism  $\tilde\phi$
of  $L^\infty(F)$ into $L^\infty(G)$. To see that $\tilde\phi$ is w*-continuous, let $\{E_i\}$ be a family 
of pairwise disjoint  Borel subsets of $F$ and $E:=\cup_iE_i$.  Then 
\begin{align*}
\tilde\phi(\sum_i\chi_{E_i}) &= \tilde\phi(\chi_{\cup_iE_i}) = \tilde\phi(\chi_E)
=\chi_{\ga^{-1}(E)} \\ &=\chi_{\cup_i\ga^{-1}(E_i)}=\sum_i\chi_{\ga^{-1}(E_i)}=\sum_i\tilde\phi(\chi_{E_i})\, .
\end{align*} 
It follows from this that for any nonnegative $h\in L^1(G)$ the positive linear 
form $u\to\dua{\tilde\phi(u),h}$ is completely additive, hence w*-continuous; thus for  any 
$h\in L^1(G)$ the linear form $u\to\dua{\tilde\phi(u),h}$  is  w*-continuous, 
and thus $\tilde\phi$ is w*-continuous.
\qed\medskip

\begin{remark}					
The assumption of w*-continuity of the extension cannot be omitted.
\end{remark} 
\noindent Indeed, let 
$\phi:  A(F)\rightarrow B(G)$ be as above with $Y\neq  G$ and 
 $\tilde {\phi}: L^\infty(F)\rightarrow L^\infty(G)$  be the map defined by
(\ref{typ2}). 
If $F$ is not compact, the ideal generated by $A(F)$ in $L^{\infty}(F)$ is proper.
Let $\delta:L^\infty(F)\to\bb C$ be a nonzero homomorphism  which vanishes on this ideal.
Then the map
$\phi_1: u\rightarrow \tilde \phi(u)+\delta(u)\chi_{G\setminus Y}$  
is a *-homomorphism from 
$L^\infty(F)$ to $L^\infty(G)$ which extends $\phi$ but does not satisfy the above formula. 

\bigskip

In this section we will prove that for a certain class of {naturally arising} continuous maps $\ga$ the condition  
$\ga_*(\mu)\ll \nu$ holds  if and only if $\ga$ is an open map.

The following is  Theorem 3.5 of \cite{aek}:

\begin{theorem}\label{25} Let $G$ and $F$  be locally compact second countable groups 
with Haar measures $\mu $ and $\nu $ respectively, and let
$\theta : G\rightarrow F$ be a continuous homomorphism.  The following are equivalent:

(i) $\theta _*(\mu )\ll\nu $

(ii) $\theta(G)$ is an open subgroup of $F$.
\end{theorem}

\begin{definition}
Let $G$ and $F$  be locally compact  groups, $K$ a subgroup of $G$  and $C$ 
a left coset of $K$  in $G$. A map $\alpha: C \rightarrow F$ is
called affine (resp. anti-affine) if there exists  a 
homomorphism  
(resp. anti-homomorphism) $\theta : K\rightarrow F$ and elements $s_0\in F, t_0\in G$ 
such that $C=t_0^{-1}K$ and 
$$\alpha(t)=s_0\theta (t_0t)$$ 
for all $t \in C=t_0^{-1}K$. 
\end{definition}

\begin{corollary}
\label{255} Let $G$ and $F$  be locally compact second countable groups with 
Haar measures $\mu $ and $\nu $ respectively. Let $C$ be a coset of an open subgroup 
$K$ of $G$ and $\alpha: C\rightarrow F$
a continuous affine or anti-affine map. The following are equivalent:

(i) The measure $\alpha_*\mu$ satisfies $\alpha_*\mu \ll\nu $. 

(ii) $\alpha(C)$ is an open subset of $F$.
\end{corollary}

\begin{proof} 
We prove the corollary for affine maps. The proof for  anti-affine maps is similar.
Write
$$\alpha=l_{s_0}\circ \theta\circ l_{t_0}, $$ 
where $l_x$ denotes left translation by $x$. 
It is clear that $\alpha(C)$ is open if and only if $\theta(K)$ is open
and  that $\alpha_*{\mu}\ll\nu $ if and only if 
$\theta_*\mu_K\ll\nu $, where $\mu_K$ is the restriction of $\mu$ to $K$.
Since $K$ is open, $\mu_K$ is a Haar measure for $K$, and Theorem \ref{25}  yields the result. 
\end{proof}
 
We shall need the  following 
definition of mixed piecewise affine maps, which is a modification of the one in \cite{spronk2}.
Note that we have included continuity in the definition of the maps. 

If $G$ is a locally compact group, the {\em open coset ring}   $\Omega_0(G)$ of $G$ 
is the smallest ring of subsets of $G$ containing the open cosets.

\begin{definition}\label{1111} 
Let $G$ and $F$ be locally compact groups and $Y$ an open and closed (clopen) subset of $G.$ 
A map $\alpha: Y\rightarrow F$ is called {\em mixed  piecewise affine} if there exist 
disjoint clopen sets $\{Y_i: i=1,...,m\}\subseteq\Omega_0(G)$ such that $Y=\cup_{i=1}^mY_i$ 
and open cosets $C_i$ of $G$ such that $Y_i\subseteq C_i$ and affine  
or anti-affine \emph{continuous} maps $\alpha_i: C_i\rightarrow F$
 such that  $\alpha|_{Y_i}=\alpha_i|_{Y_i}.$ 

If all the $\alpha_i: C_i\rightarrow F$ are affine maps, $\alpha$ is called \textit{piecewise affine}.
\end{definition} 

Recall that a second countable locally compact space is metrisable.

\begin{theorem}\label{52}
Let $G$ and $F$ be locally compact second countable groups and $Y\subseteq G$ be a  clopen set. Let
$\alpha: Y\rightarrow F$ be a mixed piecewise affine map and $m, Y_i, C_i, \alpha_i$, for
$i=1, 2, \dots, m$ be as in Definition \ref{1111}.
The following are equivalent:

\begin{enumerate}
\item $\alpha_*\mu_Y\ll \nu$
\item   the sets
$\alpha_i(C_i),  i=1,\dots, m$ are open in $F.$
\end{enumerate}
\end{theorem}

\begin{proof}
{\it (1) implies (2): }
Let $i=1,\dots,m$. We show that $\ga_i(C_i)$ is open. Since $\ga_i(C_i)$ is a coset, it suffices to show
that $\nu(\ga_i(C_i))>0$. But if $\nu(\ga_i(C_i))=0$ then $\nu(\ga_i(Y_i))=0$, so 
$\alpha_*\mu_{Y_i}(\alpha_i(Y_i))=\mu(Y_i)=0$. 
Since $Y_i$ is open, this is a contradiction.
\medskip

\noindent{\it (2) implies (1): } Suppose first that $m=1$, so that $\ga$ is the restriction of a continuous affine 
or anti-affine map.  If $\alpha(C)$ is  open, then  by Corollary \ref{255}, 
$\alpha_*\mu_{C}\ll \nu$, and a fortiori  $\alpha_*\mu_{Y}\ll \nu$. 

For the general case, suppose that the sets $\alpha_i(C_i)$ are open. By the special case just proved, we have 
${\ga_i}_*\mu_{Y_i}\ll \nu$ for each $i$.
Since $\sum {\alpha_i}_*\mu_{Y_i}=\alpha_*\mu_Y$, 
it follows that $\alpha_*\mu_Y\ll \nu$.
 \end{proof}

\begin{remark}\label{re}
Note that the condition that {\em $\alpha_i(C_i),  i=1,\dots, m$ are open in $F$}  is equivalent to the condition 
that  $\alpha$ is an open map.
\end{remark}
\noindent Indeed, let $   i \in \{1,\dots, m\}$. There exist an open subgroup
$K_i$ of $G$,  $t_i \in G   $, $s_i \in F$, and a  continuous homomorphism or
anti-homomorphism $\theta_i: K_i \rightarrow F$  such that $C_i=t_i^{-1}K_i$  and
$\alpha_i(t)=s_i\theta_i(t_it)$, for all $t \in C_i$. Since
$\alpha_i(C_i)$   is  open in $F$, the subgroup $\theta_i(K_i)$ is also open in $F$ and
it follows  by \cite[Lemma 3.4]{aek} that $\theta_i$  is an open map and hence $\alpha_i$ is an open map.

Therefore if $V\subseteq Y$ is an open set, then for each $i$ the set $\ga(V\cap Y_i)$ is open, since 
it equals  $\ga_i(V\cap Y_i)$; hence $\ga(V)=\cup_i\ga(V\cap Y_i)$ is open. \qed

\begin{theorem}\label{5555} 
Let $G$ and $F$ be locally compact second countable groups and $Y\subseteq G$ be a  clopen set. Let 
$\alpha: Y\rightarrow F$ be a mixed piecewise affine map 
(Definition \ref{1111}). The following are equivalent:
 \begin{enumerate}
\item The map $\phi_\ga:A(F)\to B(G)$ extends to a  bounded $w^*$-continuous homomorphism from 
$L^\infty (F)$ to $L^\infty (G)$. 

\item $\alpha $ is an open map.
\end{enumerate}
\end{theorem} 

\begin{proof} 

\noindent{\it (1) implies (2): }
Suppose that  the map $\phi_\alpha$  extends to a bounded $w^*$-continuous 
homomorphism from $L^\infty (F)$ to $L^\infty (G)$. For $i=1,\dots,m$ and $u\in A(F)$ let 
$$ 
\phi _i(u)(t)=\left\{\begin{array}{ll} \phi_\ga(u) (t), & t\in Y_i\\  
0, & t\in G\setminus Y_i\end{array}\right.
$$
(note that  $\phi _i$ maps $A(F)$ to $B(G)$ since the characteristic function of $Y_i$ is in $B(G)$). 
Clearly $\phi_i$ extends to a map 
from $L^\infty (F)$ to $L^\infty (G)$ and hence, by Proposition \ref{pr21},
${\alpha_i}_*(\mu_{Y_i})\ll \nu$.

Therefore, from Theorem \ref{52} and Remark \ref{re},
it follows that $\alpha$ is an open map. \smallskip

\noindent{\it (2) implies (1):}
Conversely, suppose that $\alpha$ is an open map. Then
$\alpha_*\mu_Y\ll \nu$. 

Indeed the restriction of $\ga$ to each $Y_i$ coincides with $\ga_i|_{Y_i}$.
Thus $\ga_i|_{Y_i}$ is an open map. 
But $\ga_i$ is affine or anti-affine, hence $\ga_i(x)=s\theta(tx)$ for some $s,t$, where $tC_i$ is a subgroup 
of $G$ and $\theta:tC_i\to F$ is a continuous
homomorphism (or anti-homomorphism). When $\theta$ is a homomorphism,  
$\ga_i(Y_i)=s\theta(tY_i)$ so $\theta(tY_i)$ must be open in $F$;  thus 
$\theta(tC_i)$ is a group containing a non empty open set, hence it must be open 
and so $\ga(C_i)=s\theta(tC_i)$ is open. Thus
the claim follows by Theorem \ref{52}. 
The case of an anti-homomorphism is similar.  
 
Therefore the map
$$\phi _0: L^\infty (F, \nu )\rightarrow L^\infty (Y, \mu_Y ): \;\;\;\phi _0(u)=u\circ \alpha $$
is a w*-continuous *-homomorphism. The space  $L^\infty (Y, \mu_Y )$
is a direct summand of $L^\infty (G, \mu )$ and the injection
$\iota: L^\infty (Y, \mu_Y )\rightarrow L^\infty (G, \mu )$ is $w^*$-continuous.
Hence $\tilde\phi:=\iota\circ \phi_0$  is $w^*$-continuous and clearly extends $\phi_\alpha$. 
 \end{proof}

A related result is proved by Ilie and Stokke  in \cite{IST}: if 
$Y\subseteq G$ is an open set, $\alpha: Y\rightarrow F$  a piecewise affine map and $\phi_\ga$ 
the corresponding map from $A(F)$ to $B(G)$
then $\phi_\ga$ extends to a  map $\tilde{\phi}: B(F) \rightarrow B(G)$ 
which is continuous for the respective w* topologies
if and only if $\alpha$ is an open map. 

\section{The dual problem for locally compact abelian groups} \label{3000}

Let $E$ and  $H$ be locally compact abelian groups and  $\go: E\to H$ a continuous homomorphism.
It follows from our results in Section 2 that the map $u\to u\circ\go: A(H)\to B(E)$ extends to a w*-continuous
 *-homomorphism $L^\infty(H)\to L^\infty(E)$ if and only if $\go$ is open. 
 
 In the next section, we will consider the dual problem for general locally compact groups: 
 under what conditions a homomorphism 
 $L^1(H)\to M(E)$ induces a  w*-continuous
 *-homomorphism $VN(H)\to VN(E)$. 
 
Proposition \ref{100} below examines what happens in the abelian group case. We include it  as motivation 
for our Theorem \ref{nrr} (of which of course it is a special case). It is a consequence of Theorem 
\ref{5555} and the  following  result: 

\begin{theorem}\cite[Theorem 8]{dr}\label{opr} 
Let $E$ and  $H$ be locally compact abelian groups and  $\go: E\to H$ a continuous homomorphism.
If $\hat E$ and $\hat H$ are the dual groups 
we denote by $\hat \go: \hat H\to\hat E$  the dual map 
(given by $\hat \go(\gg)=\gamma\circ\go$ for every $\gamma\in \hat H$). 

The following are equivalent:

(i) the map $\go: E\to H$ is open

(ii) the map $\hat \go: \hat H\to\hat E$ is  proper.
\end{theorem}

Below if $\theta: F\to G$ is a measurable map, 
for  $f\in L^1(F)$ we denote by $\theta_*(f)\in M(G)$ the measure given by 
$\theta_*(f)(A)=\int_{\theta^{-1}(A)}f(t)dt, \ A\subseteq G$ Borel.  
\begin{proposition}\label{100} If $F$ and $G$ are locally compact abelian groups,
a homomorphism $\phi: L^1(F)\to M(G)$ of the form $f\mapsto \theta_*(f)$, 
where $\theta: F\to G$ is a continuous group homomorphism, induces
a   w*-continuous *-homomorphism $VN(F)\to VN(G)$ if and only if $\theta$ is proper.
\end{proposition}

\begin{proof} 
Recall that, for an abelian group $F$,  the map $u\to \hat u$ (where $\hat u$ is the Fourier transform of $u$) 
is an isometric algebra isomorphism of $A(F)$ onto $L^1(\hat F)$ and
$$u(x)=\int_{\hat F}\hat u(\eta)\eta(x)d\eta$$ 
for all $x\in F$ (integration is understood with respect to suitably normalized Haar measure on $\hat F$).  
The dual map of the isomorphism  
$$L^1(\hat F)\to A(F): \hat u\mapsto u$$
is a w*-continuous isomorphism $\go^F:VN(F)\to L^\infty(\hat F)$ satisfying, for each $x\in F$,
$$ VN(F)\ni   \gl^F_x \mapsto g^F_x \in L^\infty(\hat F)$$ 
where $$g^F_x(\eta ) = \eta(x), \quad \eta\in\hat F.$$
Assume that $\theta$ is a proper homomorphism. 
Then $\hat\theta: \hat G\to \hat F$ is open, and so  by Theorem \ref{5555} the map 
$$\hat\phi: A(\hat F)\to L^\infty (\hat G):\hat\phi(u)=u\circ \hat \theta$$ extends to 
a w*-continuous map  $L^\infty (\hat F)\to L^\infty(\hat G)$ denoted by the same symbol. 
Consider the composition 
$$\psi:VN(F)\stackrel{\go^F}{\rightarrow} L^\infty(\hat F)\stackrel{\hat\phi}{\rightarrow} L^\infty(\hat G)
\stackrel{(\go^G)^{-1}}{\rightarrow} VN(G).$$ 
For each $x\in F$, this gives 
$$\psi:\gl^F_x\to g^F_x\to g^F_x\circ\hat\theta\to\gl^G_{\theta(x)}$$ 
because $g^F_x\circ\hat\theta= g^G_{\theta(x)}$;
indeed, for each $\gg \in \hat G$,
$$g_{\theta(x)}^G(\gg)=\gg(\theta(x))=g_x^F(\gg\circ \theta)=(g_x^F\!\circ\hat \theta)(\gg).$$
Thus  $$\psi(\gl^F_x)=\gl^G_{\theta(x)}$$ for all $x\in F$, and therefore, since $\psi$ is w*-continuous, 
$$\psi(\gl^F(\mu))=\gl^G(\theta_*(\mu))$$  for all $\mu\in M(F)$.
This shows that the map $$\theta_*:L^1(F)\to M(G): f\mapsto  \theta_*(f), $$
which  extends to a map  from $M(F)$ to $M(G)$,  
induces the map $\psi$, in the sense that $\psi\circ\gl^F=\gl^G\circ\theta_*$.
\smallskip
 
Conversely, assume that the map $ \theta_*: L^1(F)\to M(G)$, 
where $\theta: F\to G$ is a continuous group homomorphism, induces a $w^*$-continuous homomorphism 
$$\psi: VN(F)\to VN(G) \;\;\text{such that }\; \gl_x\mapsto \gl_{\theta(x)}  \quad  \text{for }\; x\in F.$$ 
Then the composition 
$$\hat\phi:L^\infty(\hat F)\stackrel{(\go^F)^{-1}}{\rightarrow}VN(F) \stackrel{\psi}{\rightarrow} VN(G)
\stackrel{\go^F}{\rightarrow} L^\infty(\hat G)$$ 
is a  $w^*$-continuous homomorphism which satisfies, for each $x\in F$,  
$$\hat\phi:g^F_x\to\gl^F_x\to\gl^G_{\theta(x)}\to   g^F_x\circ\hat\theta\, .$$ 
Since the w*-closed linear span of $\{g^F_x: x\in F\}$ is $L^\infty(\hat F)$, it follows
that $$\hat\phi(f)=f\circ \hat \theta$$  for all $f\in L^\infty(\hat F)$. 
Since $\hat \theta : \hat G\to \hat F$ is a continuous homomorphism we have that the restriction of 
$\hat\phi$ sends $A(\hat F)$ to $B(\hat G).$ Applying Theorem \ref{5555}, we see that the homomorphism
$$\hat \theta: \hat G\to\hat F$$ 
is open and therefore, by Theorem \ref{opr}, $\theta$ is proper.
\end{proof} 

We proved that if $F$ and $G$  are   abelian locally compact groups 
and  $\phi: L^1(F)\to M(G)$ is a homomorphism
of the form $\phi(f)=\theta_*(f)$ with $\theta$ a continuous homomorphism $F\rightarrow G$, 
then $\phi$  induces a   w*-continuous *-homomorphism $\psi: VN(F)\to VN(G)$   
if and only if $\theta$ is proper. This result suggests that a similar condition should hold for not necessarily 
abelian groups $F$ and $G$. Greenleaf \cite{green} has shown that the  form of a 
contractive homomorphism $\phi: L^1(F)\to M(G)$ is  more general than the 
 one we consider in Proposition \ref{100}. 
Nevertheless,  we prove in the 
next section that in the general case there still is  a properness condition which is 
necessary and sufficient for the existence of the induced map $\psi$.


\section{The dual problem for arbitrary locally compact groups}\label{4}

In this section we investigate conditions on a homomorphism $\phi:L^1(F)\to M(G)$ under which 
the corresponding homomorphism between the convolution algebras induces a w* continuous 
 homomorphism  $\psi:VN(F)\to VN(G)$.  
 
 We begin with some preliminaries.

Let $G$ be a locally compact group. The space $M(G)$ of complex measures on $G$ 
is the dual of $C_0(G)$; the duality is given by 
$$\dua{\mu,f}=\int_Gfd\mu,\quad f\in C_0(G),\,\mu\in M(G).$$
The space $M(G)$ is a Banach algebra under the following multiplication
$$\dua{\mu\ast \nu,f}=\int_G\int_Gf(ts)d\mu(t)d\nu(s),\quad f\in C_0(G).$$
We say that the net $(\mu_i)\subseteq M(G)$ converges in the $so$ topology to $\mu \in M(G)$ if  
$$\nor{\mu_i\ast f- \mu\ast f}_{L^1}\to 0  \quad \text{for all }\,f\in L^1(G). $$

We will need the following simple facts:

\begin{lemma}\label{di}
 The map $\gl: M(G)\to VN(G)$ is the dual of the inclusion $\iota:A(G)\to C_0(G)$.
\end{lemma}
 
\begin{proof}  An element $u\in A(G)$ is a vector functional on $VN(G)$, i.e. an element of the predual
$(VN(G))_*$ given by 
$\dua{T, u}= \sca{T\xi,\eta}$ for some $\xi, \eta\in L^2(G)$. In particular 
$$\dua{\gl(\mu),u}=\sca{\gl(\mu)\xi,\eta} = \int_G\sca{\gl_t\xi ,\eta }d\mu(t).$$
But the functional $u$ defines a function $\iota(u)\in C_0(G)$, by 
$\iota(u)(s):=\dua{\gl_s,u}= \sca{\gl_s\xi,\eta}$ (usually one denotes the function $\iota(u)$ by $u$,
identifying the functional with the function). We have 
\begin{align*} 
\dua{\mu,\iota(u)}= \int \iota(u)(t)d\mu(t)=\int_G\sca{\gl_t\xi ,\eta }d\mu(t)=\dua{\gl(\mu),u}\, .
\end{align*}
This relation holds for all $u\in A(G)$ and all $\mu\in M(G)$ and shows that $\gl=\iota^*$. 
\end{proof}

\begin{proposition}\label{z} Let $F$ be a locally compact group. Then
the maps 
$$\gl: (M(F), w^*)\to (VN(F), w^*) \quad \text{and }\quad \gl: (M(F), so)\to (VN(F), w^*)\,$$  are  continuous.
\end{proposition}

\begin{proof} 
The first assertion  follows from Lemma  \ref{di}.

For the second assertion, since the identity map $\iota: M(F)\to M(F)$ is $so\text{-}w^*$-continuous 
\cite[Proposition 2.2]{sto2011}, the composition $\gl=\gl\circ\iota: M(F)\to VN(F)$ is $so\text{-}w^*$-continuous.
\end{proof}

The proof of the following Proposition is contained in \cite[Corollary 3.11]{IS}.  
\begin{proposition}\label{xxxxxx} Let $F$ and $G$  be locally compact groups, $\theta: F\rightarrow G$ 
a continuous homomorphism and 
$\rho(u)=u\circ\theta$ for every $u\in A(G).$ Then: $\theta$ is proper if and only if   $\rho(A(G))\subseteq A(F).$
\end{proposition}

\begin{proposition}\cite[Proposition 2.2]{sto2011}
Any  bounded homomorphism $\varphi_0:L^1(F)\to M(G)$ has a unique $so$-$w^*$-continuous 
extension to a homomorphism $\varphi: M(F)\rightarrow M(G)$.
\end{proposition}	
Thus in the sequel it will be enough to consider $so$-$w^*$-continuous  homomorphisms between the 
measure algebras. 
\begin{definition} Let $F$ and $G$  be locally compact groups 
and $$\varphi: M(F)\rightarrow M(G)$$ be a $so$-$w^*$-continuous  homomorphism. 
We say that $\varphi$ has the {\em extension property} if there exists a $w^*$-continuous 
homomorphism $\psi: VN(F)\rightarrow VN(G)$ such that $\psi\circ\lambda^F=\lambda^G\circ \varphi.$
\end{definition}

Note that it follows from \cite[Proposition 5.5]{sto2011} that a $so$-$w^*$ continuous, contractive  homomorphism 
$\varphi:M(F)\rightarrow M(G) $ is a $*$-homomorphism.

\medskip

For clarity of exposition, we begin with the case of a unital  homomorphism, since in this case the 
construction of the `symbol' $\theta$ is more direct.  

\subsection{The unital case} 

In this subsection we assume that  $\varphi:M(F)\rightarrow M(G) $ is a contractive unital $*$-homomorphism. 

\begin{lemma}\label{300} Let $F$ and $G$  be locally compact groups and 
$$\varphi: M(F)\rightarrow M(G)$$ be a $so$-$w^*$-continuous contractive unital $*$-homomorphism 
$(\varphi(\delta_{e_F})=\delta_{e_G})$.
Then there exist, necessarily unique, continuous homomorphisms 
$$\theta: F\rightarrow G, \;\beta: F\to \bb T$$ such that 
$$\varphi(\delta^F_x)=\beta(x)\delta^G_{\theta(x)}, \;x\in F.$$ 
\end{lemma}
\begin{proof} 
Applying \cite[Theorem 3.1.8]{green} to the group $\{\varphi(\delta_x^F)\}\subseteq \mathop{\rm ball }(M(G))$, 
for each $x\in F$ we have $\varphi(\delta_x^F)=b\delta^G_g$  for some $b\in\bb T$ and $g\in G$. 
We define maps
$$\theta: F\to  G, \quad\beta: F\to \bb T$$ by the following rule
$$\theta(x)=g \text{ and } \beta(x)=b\;\Leftrightarrow\; \varphi(\delta_x^F)=b\delta^G_g\, .$$
Using the fact that $\varphi$ is unital and multiplicative, it is easily checked that  
$\theta$ and $\beta$ are well defined and are group homomorphisms.
We have
$$\varphi(\delta^F_x)=\beta(x)\delta^G_{\theta(x)},\; \; x\in F.$$  
Since $\varphi$ is $so$-$w^*$-continuous the homomorphisms  $\theta$ and $\beta$ are continuous. 
 This follows from  \cite[Lemma 4.1]{sto2011}. We include a short proof for completeness:

 Assume that $x_i\rightarrow x.$ We shall prove that $\beta(x_i)\rightarrow \beta(x)$ and 
$\theta(x_i)\rightarrow \theta(x).$ 
Since $w^*\text{-}\lim\delta_{x_i}^F=\delta_{x}^F$,  by  \cite[Lemma 1.1.2]{green}
we have $so$-$\lim\delta_{x_i}^F=\delta_{x}^F$
and hence $w^*$-$\lim\varphi(\delta_{x_i}^F)=\varphi(\delta_{x}^F)$. Thus $$w^*\text{-}\lim_i\beta(x_i)\delta^G_{\theta(x_i)}=\beta(x)\delta^G_{\theta(x)}.$$

Take a subnet $(\delta_{\theta(x_j)})$  of $(\delta_{\theta(x_i)})$ which converges, say to $\mu.$ 
Since $\{\beta(x_j)\}$ lies in a compact set, passing to a subnet if necessary we may assume that 
$\beta(x_j)\to b$ for some $b\in\bb T$ .
We have 
$$\mu=w^*\text{-}\lim_j\delta_{\theta(x_j)}=
w^*\text{-}\lim_j\frac{1}{\beta(x_j)}\beta(x_j)\delta_{\theta(x_j)}=\frac{1}{b}\beta(x)\delta_{\theta(x)}.$$ 
Since $\mu$ is a positive measure, being the limit of $(\delta_{\theta(x_j)})$,
we have $$\beta(x)=b\;\text{and }\;  \delta_{\theta(x)}=\mu.$$
We proved that every convergent subnet of $(\delta_{\theta(x_i)})$  converges to $ \delta_{\theta(x)}.$ 
It follows that $w^*\text{-}\lim\delta_{\theta(x_i)}=\delta_{\theta(x)}$ and hence 
$$\theta(x_i)\rightarrow \theta(x).$$
A similar argument shows that   
$\beta(x_i)\to \beta(x).$   
\end{proof} 

\begin{theorem}\label{400} Let $F$ and $G$  be locally compact groups and 
$$\varphi: M(F)\rightarrow M(G)$$ be a unital contractive $so$-$w^*$-continuous homomorphism.
The following are equivalent: 

(i) The map $\varphi$ has the extension property.

(ii) The homomorphism $\theta$ is proper. 
\end{theorem}

\begin{proof} $(i)\Rightarrow (ii)$ \\
Suppose there exists a w*-continuous map $\psi: VN(F)\rightarrow VN(G)$  such that 
$\psi\circ\lambda^F=\lambda^G\circ \varphi$. 
For all $x\in F$ we have 
$$\psi(\lambda_x^F)=\psi\circ\lambda(\delta_x^F)=\lambda\circ\varphi (\delta_x^F)=
\lambda(\beta(x)\delta^G_{\theta(x)})=\beta(x)\lambda^G_{\theta(x)}.$$
Therefore if $\psi_*: A(G)\rightarrow A(F)$ is the predual of $\psi$ then for all $u\in A(G)$ we have 
$$\dua{\gl_x^F,\psi_*(u)}=\dua{\psi(\gl_x^F),u}=\dua{\beta(x)\gl^G_{\theta(x)},u}=\beta(x)(u\circ\theta)(x)$$ 
and thus the map $x \to \beta(x)(u\circ\theta)(x)$ is  in $A(F)$.  
Since the map $x\to \overline{\beta(x)}$ is in the Fourier-Stieltjes algebra $B(F)$ 
($\overline\beta$ is a representation) it follows that  $u\circ\theta \in A(F).$
By Proposition \ref{xxxxxx}, $\theta$ must be a proper homomorphism.   

\medskip

\noindent $(ii)\Rightarrow(i)$ \\
Conversely assume that  $\theta$ is a proper homomorphism. 
Then for  all $u\in A(G)$ the map $u\circ\theta$ is in $A(F)$ and so is the map 
$F\to\bb C:x\mapsto\beta(x)(u\circ\theta)(x)$. Hence we may define 
$$\rho: A(G)\rightarrow A(F)\;\text{by}\;\rho(u)(x)=\beta(x)(u\circ\theta)(x),\,\,x\,\in\,F.$$
Then $\psi=\rho^*: VN(F)\rightarrow VN(G)$ is a w* continuous *-homomorphism satisfying 
$$\psi(\lambda_x^F)=\beta(x)\lambda^G_{\theta(x)}.$$ 
Thus  we have 
$$(\psi\circ\gl^F)(\delta_x^F)=\psi(\lambda_x^F)=\beta(x)\lambda^G_{\theta(x)}=
\lambda^G(\beta(x)\delta^G_{\theta(x)})=(\gl^G\circ\varphi)(\delta_x^F)$$
for all $x\in F$.\\
Since the map $\gl^F: M(F)\to VN(F)$ is $so$-$w^*$-continuous, $\psi\circ\gl^F$ is  $so\text{-}w^*$-continuous;
also, $\varphi:M(F)\to M(G)$  is $so$-$w^*$-continuous  and  $\gl^G: M(G)\to VN(G)$ is $w^*$-continuous, 
hence  $\gl^G\circ\varphi$ is also $so$-$w^*$-continuous. But the convex hull of 
$\{b\delta_x^F: b\in\bb T, x\in F\}$ is $so$-dense in $\mathop{\rm ball} M(F)$ \cite[Lemma 1.1.3]{green}. 
Therefore from the previous equality we obtain
$\psi\circ\lambda^F=\lambda^G\circ \varphi.$
\end{proof}

\subsection{The general case } 
In order to deal with a not necessarily unital contractive *-homomorphism, we will need some preparation.

\begin{remark}\label{idem} 
Let $G_0$ be a locally compact group and $K\subseteq G_0$ a normal compact subgroup
with  normalized Haar measure $m_K$. If $\eps: K\to\bb T$  is a continuous character of $K$, the measure 
$\eps m_K$ defines a contractive idempotent element of $M(G_0)$ (this follows easily from the definitions). 
Hence  $P_0:=\gl^{G_0}(\eps m_K)$ is an orthogonal projection in $VN(G_0)$.
\end{remark}
	
\begin{lemma}\label{ker} Let $G_0$ be a locally compact group and $K\subseteq G_0$ a normal compact subgroup
with  normalized Haar measure $m_K$. If $\pi: G_0\to G_0/K$  is the quotient map and $Q_0\in VN(G_0)$  
denotes the projection $Q_0=\lambda^{G_0}(m_K)\in VN(G_0)$,
there is a *-isomorphism $j_0: VN(G_0)Q_0\to VN(G_0/K)$ such that \\
$j_0(\gl_g^{G_0}Q_0)=\gl_{\pi(g)}^{G_0/K}$.
\end{lemma}
\begin{proof} By  \cite[Proposition 3.25]{eym} the dual of the
map $u\mapsto u\circ \pi: A(G_0/K)\to A(G_0)$ is a w*-continuous onto *-homomorphism 
$$\psi: VN(G_0)\to VN(G_0/K)$$ satisfying $\psi(\lambda_g^{G_0})= \gl_{\pi(g)}^{G_0/K}$ for all $g\in G$. 

The kernel of this map is the ideal $J=VN(G_0)Q^\bot$ where $Q$ is the projection onto the Hilbert space
$$L^2_K(G_0)=\{f\in L^2(G_0): f(kx)=f(x),\,\,\forall\,x\in G_0,\,k\,\in K\}.$$
It follows that $\psi$ induces an isomorphism from $VN(G_0)/J = VN(G_0)Q$ onto $VN(G_0/K)$ satisfying
$\gl_g^GQ \mapsto  \gl_{\pi(g)}^{G_0/K}$ for all $g\in G$. This map is the claimed isomorphism $j_0$,
because the projection $Q$ is in fact $Q_0$.

Indeed,  by \cite[3.2]{eym}, the space $L^2_K(G_0)$ coincides with  the space of all
$f\in L^2(G_0)$ satisfying $f=f*m_K=m_K*f$
since $K$ is normal, equivalently $f=\gl(m_K)f$ or $f=Q_0(f)$. Thus $Q=Q_0$ and we are done.
\end{proof}

\begin{lemma}\label{SWSTOS} Let $G_0$ be a locally compact group, $K$  a compact normal subgroup of $G_0$,
$\varepsilon$ a continuous character of $K$,  $m_K$ the normalized Haar measure of $K$ and 
$$A_\varepsilon=\{u\in A(G_0): u(gk)=u(g)\varepsilon ^{-1}(k): \forall g\in G_0, k\in K\}.$$
Then there exists an isometric surjection 
$j_\eps: VN(G_0)\gl(\varepsilon m_K)\rightarrow A_\varepsilon^*$ such that
$$\dua{j_\eps(\lambda_g\lambda(\varepsilon m_K)),u}=u(g)\quad \text{for all }\; u\in A_\varepsilon, \ g\in G_0.$$
\end{lemma}
\begin{proof} Let $P_0=\gl(\varepsilon m_K) $. 
Given $u\in A(G_0)$, the map $$\psi(u):VN(G_0)\to\bb C: T\mapsto \dua{TP_0,u}$$ is w*-continuous, 
hence defines an element of $A(G_0)$. 

Thus the map $\psi:A(G_0)\to A(G_0)$  is the predual of the map $T\mapsto TP_0$.

We claim that for all $g\in G_0$, 
$$\psi(u)(g)=\int_K u(gk)\varepsilon (k)d m_K(k)\, .$$
Indeed,  for all $g\in G_0$,
\begin{align*}
\dua{\gl_gP_0, u} &=\dua{\gl(\delta_g)\gl(\varepsilon m_K), u} 
= \dua{\gl(\delta_g*\varepsilon m_K), u} \\
& = \int_{G_0} u(t)d(\delta_g*\varepsilon m_K)(t) = 
\iint u(ts)d\delta_g(t)d(\varepsilon m_K)(s) \\
&=\int_K u(gk)\varepsilon (k)d m_K(k) \, . 
\end{align*}
But if $u\in A_\varepsilon$ then this integral equals $u(g)$,  
hence $\dua{\gl_gP_0, u}=\dua{\gl_g, u}$ and so 
 $$\dua{TP_0, u}=\dua{T, u}$$ for all $T \in VN(G_0)$ and $u \in A_\varepsilon$.

It follows that the map $j_\eps$ that sends $TP_0\in A(G_0)^*$ to its restriction to $A_{\eps}$ is onto.

We show that it is injective. If 
$\dua{TP_0, u}=0$ for all $u \in A_{\epsilon}$ then, for all $v \in A(G_0)$  
$$\dua {TP_0, v}=\dua {(TP_0)P_0, v}=\dua{TP_0, \psi(v)}=0$$ (since $P_0$ is idempotent)
and hence $TP_0=0$. 
\end{proof}  

In this subsection we assume that  $\varphi:M(F)\rightarrow M(G) $ is a so-w*-continuous
contractive *-homomorphism. 
The following can be found in \cite[Section 4.2]{green} and \cite[Theorem 3.1.8]{green}: 
Let    
$$\Gamma=\{\varphi(\delta_x^F): x\in F\}.$$ 
This is a subgroup of $\mathop{\rm ball } (M(G))$ under convolution. Its unit  is $\varepsilon m_{K}$
where $m_K$ is the normalized Haar measure of a compact subgroup $K$ of $G$,  
viewed as an element of $M(G)$ and $\eps$ is a continuous character of $K. $
Further, there exists a subgroup 
$G_{00}$  of $G$ such that for all $x\in F$ there exists $(b,g)\in \bb T\times G_{00}$ such that 
$$\varphi(\delta_x^F)=b\delta_g\ast \varepsilon m_{K}.$$
The group $K$ is a normal subgroup of $G_{00}$, hence of its closure  $G_0:=\overline{G_{00}}.$   
\\
The group $G_0$ is a locally compact subgroup of $G$.  \smallskip

The proof of the following Lemma can be deduced from \cite[Lemma 5.2]{sto2011} or 
\cite[Proposition 3.1.6]{green}. 

\begin{lemma}\label{y}
(i) For every $g\in G_0$ we have $\varepsilon m_{K}\ast\delta_g=\delta_g\ast \varepsilon m_{ K}.$

(ii)  If $x\in K$ then $$\varepsilon(x)\delta_x\ast \varepsilon m_K=\varepsilon m_K.$$
\end{lemma}

\medskip 

In what follows we will use the notation  
$$G_{00}'=\{(b,g)\in\bb T\times G_{00}: b\delta_g\ast \eps m_K\in \Gamma\}$$
$$G^\prime=\bb T\times G, \quad G_0^\prime=\overline{ G_{00}^\prime}, $$
$$K'=\{(\eps(g), g): g\in K\}.$$ 
(see also \cite[Lemma 3.1.11]{green}) and 
denote by $m_{K'}$ the normalised Haar measure of the compact group $K',$ viewed as an element of $M(G_0')$.
Greenleaf \cite[Theorem 3.1.10]{green}  proves that $K'$ is a normal subgroup of  $G_{00}^\prime$
and, if $\pi^\prime: G_0^\prime \rightarrow  G_0^\prime /K^\prime$ is the quotient map,
there exists a continuous homomorphism 
$$\theta^\prime : F\rightarrow G_0^\prime /K^\prime$$ satisfying 
\begin{equation}\label{thetapr}
\theta^\prime(x)=\pi^\prime(b,g)\iff\varphi(\delta_x^F)=b\delta^G_g\ast \varepsilon m_K.
\end{equation}
Since the projection $(b,g)\mapsto g$ maps $K'$ to $K$, it induces a continuous homomorphism
$$\tilde \nu: G_0' /K'\to G_0/K:\pi'(b,g)\mapsto\pi(g)$$
where $\pi: G_0\to G_0/K$ is the quotient map .

 The map $\tilde \nu$ is proper: Indeed, if
$L\subseteq G_0/K$ is a compact set then
$$\tilde \nu^{-1}(L)\subseteq \{\pi'(b,g): \pi(g)\in L\}\subseteq\bb T\times \pi^{-1}(L).$$ 
Since $\pi$ is proper and $\tilde \nu^{-1}(L)$ is  a closed subset of  $\bb T\times \pi^{-1}(L)$ 
we conclude that $\tilde \nu^{-1}(L)$ is a compact set.

Define the  map
$$\theta: =\tilde\nu\circ\theta':F\rightarrow G_0/K \, $$ and note that clearly
\begin{equation}\label{theta}
\theta(x)=\pi(g)\iff \varphi(\delta_x^F)=b\delta^G_g\ast\eps m_K  
\end{equation} for some $b\in\bb T$. 

Note that $\theta=\tilde\nu\circ\theta'$  is continuous, since by   \cite[Proposition 4.2.1]{green}
the map $\theta'$ is continuous.

\begin{lemma}\label{tessera}
 The map $\theta $ is proper if and only if $\theta'$  is proper.
   \end{lemma}
\begin{proof}
 If $\theta'$ is proper, then $\theta$ is proper as the composition $\theta=\tilde\nu\circ\theta'$
 of proper maps.
 
Assume that $\theta$ is proper. Let $L\subseteq  G_0' /K'$ be a compact set. Then 
$\theta'^{-1}(L)\subseteq \theta'^{-1}(\tilde \nu^{-1} (\tilde \nu(L))=\theta^{-1}(\tilde \nu(L)))$. 
Since $\theta$ is proper and $\tilde\nu$ is continuous, $\theta^{-1}(\tilde \nu(L))$ is compact
and hence its closed subset $\theta'^{-1}(L)$ is compact. 
\end{proof}

We will use the following  \smallskip

\noindent{\bf Notation } $P_0=\gl^{G_0}(\eps m_K)\in VN(G_0)$,  
and $Q'_0=\gl^{G_0'}(m_{K'})\in VN(G_0')$. 
Note that these are central projections (Remark \ref{idem} and Lemma \ref{y}(i)).

\begin{lemma}\label{l}
There exists a w*-continuous *-homomorphism 
\begin{align}\label{psi} 
j' : VN(G_0')Q'_0\to &VN(G_0)P_0\\
 \text{such that }\quad  \gl_{(b,g)}^{G_0'}Q'_0\mapsto & b\gl_g^{G_0}P_0\, , \; (b,g)\in G_0' \, . \nonumber
\end{align}
\end{lemma}
\begin{proof}
 Let us write $t=(\beta(t), g(t))$ for the elements of $G_0$. 
For $u\in A(G_0)$ define 
\begin{equation}\label{na}
\sigma(u)(t)=\dua{\gl(\beta(t)\delta_{g(t)}\!*\!\varepsilon m_K),u}
=\beta(t)\dua{\gl_{g(t)}^{G_0}P_0,u}.
\end{equation}
In Lemma \ref{SWSTOS} we have seen that for all $u\in A(G_0)$ the map 
$$G_0\to \bb C: g \mapsto \int_K u(gk)\epsilon(k)dk= \dua{\gl_{g}^{G_0}P_0,u}$$ 
is in $A(G_0)$. Hence by \cite[Proposition 3.25(i)]{eym}, the map $(b, g)\mapsto  \int_K u(gk)\eps(k)dk$ 
is in $A(G'_0)$. 
The function $\beta$ belongs to the Fourier - Stieltjes algebra $B(G'_0)$
because  $t\to\beta(t)$ is a representation; 
this shows  that $\sigma(u)$ is in the  Fourier algebra $A(G'_0)$.

\smallskip 
Rewriting relation (\ref{na}) in the form 
$$\dua{\gl_{(b,g)},\sigma(u)}=\dua{b\gl_{g}^{G_0}P_0,u}$$
we see that the dual of  the map
$$\sigma: A(G_0)\rightarrow  A(G_0'): u\rightarrow \sigma(u)$$ 
is a continuous map satisfying 
$$\sigma^*: VN(G_0')\to VN(G_0): \gl_{(b,g)}^{G_0'}\mapsto b \gl_g^{G_0}P_0.
$$

We claim that 
$$P_0=\gl^{G_0}(\varepsilon m_K)= \sigma^*(\gl^{G_0'}(m_{K'}))=\sigma^*(Q_0').$$
Indeed, note that  $m_{K'}=m_K\circ \iota^{-1}$ where $\iota: K\to K'$ is the group homomorphism 
$g\mapsto (\eps(g),g)$. 
Suppose that $m_K$ is a w*-limit of the form
$$m_K=\lim_j\sum_{i=1}^{n_j}a_i^j\delta_{g_i^j}^K.$$
Then  
$$\varepsilon m_K=\lim_j\sum_{i=1}^{n_j}a_i^j\varepsilon (g_i^j)\delta_{g_i^j}^K \quad 
\text{and }\quad m_{K^\prime}=\lim_j\sum_{i=1}^{n_j}a_i^j\delta^{K^\prime}_{(\varepsilon (g_i^j),g_i^j )}.$$
Thus 
\begin{align*}
\gl^{G_0}(\eps m_K)&=\lim_j\left(\sum_{i=1}^{n_j}a_i^j\eps(g_i^j)\lambda_{g_i^j}^{G_0}\right)\\
&=\lim_j\left(\sum_{i=1}^{n_j}a_i^j\eps(g_i^j)\lambda_{g_i^j}^{G_0}\right)\gl^{G_0}(\eps m_K) \quad
\text{since $\lambda^{G_0}(\eps m_K)$ is idempotent} \\
&=\lim_j\sum_{i=1}^{n_j}a_i^j\sigma^*(\lambda^{G_0'}_{(\varepsilon (g_i^j),g_i^j )})\\
&=\lim_j\sigma^*\left(\sum_{i=1}^{n_j}a_i^j\gl^{G_0'}_{(\eps(g_i^j),g_i^j )}\right)=\sigma^*(\gl^{G_0'}(m_{K'}))
\end{align*}
which proves the Claim.

It follows that  the linear map \begin{align*} 
j' : VN(G_0')Q'_0\to VN(G_0)P_0:\;  \gl_{(b,g)}^{G_0'}Q'_0\mapsto & \sigma^*(\gl_{(b,g)}^{G_0'})\, , \; (b,g)\in G_0' \, . 
\end{align*}
is w*-continuous. That it is a *-homomorphism follows from (\ref{psi}), 
since $Q_0'$ and $P_0$ are central projections.
 \end{proof}

\begin{lemma}\label{proj} 
Let $VN_{G_0}(G) = \overline{{\rm span}\{\gl^G_g:g\in G_0\}}^{w*}\subseteq \cl B(L^2(G))$. The map 
$$\gamma:  VN(G_0)\to VN_{G_0}(G): \gl^{G_0}_g\mapsto \gl^G_g$$ 
extends to a w*-continuous *-isomorphism mapping the central 
projection  $P_0=\gl^{G_0}(\eps m_K)$ to the central  projection $P=\gl^G(\eps m_K)\in VN_{G_0}(G)$.  

\end{lemma}
\begin{proof} 
The first statement is a  Corollary of Herz' restriction Theorem 
\cite[Proposition 2.6.6]{kanlau}. 

The fact that $P$ and $P_0$ are central projections follows from Lemma \ref{y} (i).  

 The image of $P_0=\gl^{G_0}(\eps m_K)$ under this map is clearly 
$P=\gl^G(\eps m_k)$. 

Indeed, since $K\subseteq G_0$, the measure $\eps m_K$ is supported in $G_0$, so 
$P=\int\gl^G_gd(\eps m_K)(g)$ is in $VN_{G_0}(G) $ and hence
$$\gamma(P_0)=\int\gamma(\gl^{G_0}_g)d(\eps m_K)(g)=\int\gl^G_gd(\eps m_K)(g)=P\, .$$ 
\end{proof}

\begin{theorem}\label{nrr} 
Let $F$ and $G$  be locally compact groups and 
$$\varphi: M(F)\to M(G)$$ be a contractive $so$-$w^*$-continuous
homomorphism.
The following are equivalent:

(i) The map $\varphi$ has the extension property.

(ii) The homomorphism $\theta: F\to G_0/K$ (see (\ref{theta})) is a proper map.

(iii)  The map $\varphi$ is w*-continuous. 
\end{theorem}
\begin{proof}
$(i)\Rightarrow (iii)$ \\
If $\varphi$ has the extension property, there exists a $w^*$-continuous homomorphism
$$\psi: VN(F)\to VN(G)$$  satisfying 
$$\lambda^G\circ \varphi=\psi \circ \lambda^F.$$
To show that $\varphi$ is w*-continuous, it suffices (by the Krein-Smulian Theorem)
to prove that if $(\mu_i)\subseteq {\rm ball}(M(F))$ is a net with 
w*-$\lim_i\mu_i=0$ then w*-$\lim_i\varphi(\mu_i)=0.$ 

Since $\gl^G\circ \varphi: M(F)\to VN(G)$ is w*-w*-continuous, we have that $\lim_i\gl^G\circ \varphi(\mu_i)=0$ 
in the w*-topology of $VN(G)$, i.e. that $\dua{\gl^G\circ \varphi(\mu_i),f}\to 0$ for all $f \in A(G)$.
By Lemma \ref{di}, this means that 
$$\int_Gfd\varphi(\mu_i)\to 0,\,\,\forall\,f\in A(G).$$
Since the net $(\varphi(\mu_i))$ is bounded and $A(G)$ is uniformly dense in $C_0(G)$,
it follows that $$\int_Gfd\varphi(\mu_i)\to 0,\,\,\forall\,f\in C_0(G).$$
Thus w*-$\lim_i\varphi(\mu_i)=0.$ 
\medskip

\noindent $(iii)\Rightarrow (ii)$

Assume that $\varphi$ is $w^*$ continuous. Consider the *-homomorphism 
$$\psi_0:=\gl^G\circ \varphi: M(F)\to VN(G).$$
By Proposition \ref{z}, this map is w*-continuous. By (\ref{thetapr}), for all $x\in F$ we have  
$$\psi_0(\delta^F_x)=\gl^G(b\delta_g*\eps m_k)=b\gl_g^GP \quad\text{when }\;\pi'(b,g)=\theta' (x)$$
(recall  that $P=\gl^G(\eps m_K)$). 

Note that $\psi_0$ takes values of the form $b\gl^G_gP$ where $g\in G_0$, hence 
$$\psi_0(M(F))\subseteq VN_{G_0}(G)P\subseteq VN_{G_0}(G).$$
By Lemma \ref{proj} there is a w*-continuous isomorphism $\gamma^{-1}:  VN_{G_0}(G)\to VN(G_0)$ such that 
\begin{align*}
\gamma^{-1}(\gl_g^GP)=\gl_g^{G_0}P_0
\end{align*}
(where $P_0=\gl^{G_0}(\varepsilon m_K)$).
Thus the composition 
$$\zeta:=\gamma^{-1}\circ\psi_0: M(F)\to  VN(G_0)P_0$$ 
is a well-defined w*-continuous *-homomorphism satisfying
$$\zeta(\delta_x^F)= b\gl_g^{G_0}P_0 \;\;\text{when }\;\pi'(b,g)=\theta^\prime (x).$$
Recall (Lemma \ref{SWSTOS}) that the predual of $VN(G_0)P_0$ is $A_\varepsilon$. The predual of  $\zeta$  is
the map $\zeta_*: A_\varepsilon \rightarrow C_0(F)$ which satisfies, for all $u\in A_\eps$ and $x\in F$,   
$$\zeta_*(u)(x)= \dua{\zeta(\delta_x^F),u} =\dua{b\gl_g^{G_0}P_0,u}\quad\text{when }\;\pi'(a,g)=\theta' (x).$$
But again by Lemma \ref{SWSTOS} we have 
$$\dua{\gl_g^{G_0}P_0,u }=u(g)$$
for $g\in G_0$; thus   for all $u\in A_\varepsilon$ and $x\in F$,  
\begin{align}\label{zeta}
			\zeta_*(u)(x)=b u(g)  \;\text{ when }\; \pi'(b,g)=\theta'(x).
\end{align}
Now suppose,  by way of contradiction, that $\theta$ is not proper. Then there is a compact set 
$\tilde L\subset G_0/K$ such that $ \theta ^{-1}(\tilde L)$ is not compact. Let $L=\pi^{-1}(\tilde L).$ 
Observe that $L$ is compact in $G_0.$

Fix a nonzero $u \in A_{\varepsilon}$ and  $g_0 \in G_0$ such that $|u(g_0)|= 2\delta >0$. 
For each $h \in G_0$ the function $u_h(s):=\lambda(hg_0^{-1})u(s)$ is in $A_\varepsilon$
and  $|u_h(h)|=|\lambda(hg_0^{-1})u(h)|=|u(g_0)|=2\delta$. Hence there is an open set $V_h$ containing $h$ 
such that $|u_h(k)|\geq \delta >0$ for each $k \in V_h$.
Since $\{V_h: h\in L\}$ is an open cover of the compact set $L$, taking a finite subcover we obtain
a finite number of functions
$u_1,\dots ,u_n\in A_\varepsilon$  such that if $$L_i=\{g\in L: |u_i(g)|\ge \delta \}$$ then $L=\cup_{i=1}^n L_i.$ 

If $x\in  \theta ^{-1}(\tilde L)$ then $\theta(x)\in \tilde L$ and thus $\pi^{-1}(\{\theta(x)\})\subseteq L$.
Therefore there exists $g\in L_i$ such that $\pi(g)=\theta(x)$, 
hence  $\pi'(b,g)=\theta'(x)$ for some $b\in\bb T$. 
 Thus, by (\ref{zeta}),
$$|\zeta_*(u_i)(x)|\ge \delta.$$ Hence if we let 
 $$M_i=\{x\in \theta ^{-1}(\tilde L): |\zeta_*(u_i)(x)|\ge \delta\},\quad1\le i\le n$$ 
 we have
  $$\theta ^{-1}(\tilde L)=\bigcup _{i=1}^nM_i\, .$$
 Since $\theta ^{-1}(\tilde L)$ is not compact, some set $M_i$ is not compact. 
But  
 $|\zeta_*(u_i)(x)|\ge \delta$ for all $x \in M_i$, which contradicts the fact that $\zeta_*(u_i)\in C_0(F)$.
 
\medskip\noindent  $(ii)\Rightarrow (i)$
 
 Assume that $\theta$ is a proper homomorphism.  Then  $\theta'$ is proper by Lemma \ref{tessera}.
By Proposition \ref{xxxxxx}, the map $u\mapsto u\circ\theta'$ sends $A(G_0'/K')$ to $A(F)$ and so
its dual is a $w^*$-continuous homomorphism 
$\psi_1: VN(F)\rightarrow VN(G_0'/K')$ sending 
$\lambda_x^F$ to $\lambda_{\theta'(x)}^{G_0'/K'}.$ 
Recall (Lemma \ref{ker}) that $VN(G_0'/K')$ is isomorphic to $VN(G'_0)\gl^{G_0'}(m_{K'}).$  
Thus we may assume that 
$$\psi_1: VN(F)\to VN(G'_0)\gl^{G'_0}(m_{K'})$$ sends  
$\gl_x^F$ to $\gl^{G_0'}_{(b,g)}\gl^{G_0'}(m_{K'})$  when $\pi'(b,g)=\theta'(x)$.
We also recall the $w^*$ continuous homomorphism  (\ref{psi})  
\begin{align*}
j': VN(G_0')\gl^{G_0'}(m_{K'})\to &VN(G_0)\gl^{G_0}(\varepsilon m_K)=VN(G_0)P_0\\
 \gl_{(b,g)}^{G_0'}\gl^{G_0'}(m_{K'})\mapsto & b\gl_g^{G_0}P_0\, . 
\end{align*}
Composing with the {w*-continuous *-isomorphism 
$\gamma: VN(G_0)\to VN_{G_0}(G): \lambda_g^{G_0}\mapsto \gl_g^{G}$ (which sends $P_0$ to $P$)
we have a *-homomorphism
 \begin{align*}
\gamma\circ j': VN(G_0')\gl^{G_0'}(m_{K'})\to &VN(G)P\\
\gl_{(b,g)}^{G_0'}\gl^{G_0'}(m_{K'}))\mapsto & b\gl_g^{G}P\, . 
\end{align*}}
Finally, composing with $\psi_1$, we obtain a  w*-continuous *-homomorphism
$$\psi:=\gamma\circ j'\circ\psi_1: VN(F)\rightarrow VN(G)$$
sending $\gl_x^F$  to $b\gl^G_gP$ 
when $\varphi(\delta_x^F)=b \delta_g^{G_0}*\varepsilon m_K$. 

 From the construction of $\psi$ we have  
$$(\gl^G\circ\varphi)(\delta_x^F)=b\gl_g^G\gl^G(\eps m_K)=\psi(\gl^F_x)=(\psi\circ\gl^F)(\delta^F_x)$$ 
for all $x\in F$. Now the same density argument as in the conclusion of the proof of Theorem \ref{400} shows that 
$(\gl^G\circ\varphi)(\mu)=(\psi\circ\gl^F)(\mu)$  
for all $\mu\in M(F)$, 
which means that  $\varphi$ has the extension property.
\end{proof}

\begin{remark}
The main focus in this Section is the equivalence of the extension property with the properness of the map 
$\theta$. Note that the equivalence $(ii)\iff (iii)$ is also proved in \cite[Proposition 3.4]{sto2012}. 
Our approach is different, as it 
relies on operator algebra methods and exploits the duality between 
the Fourier algebra and the von Neumann algebra of a group; it 
highlights the use of the extension property.
\end{remark}


\end{document}